\documentclass[12pt]{amsart}
\usepackage{amsmath,amssymb,latexsym, amsthm, amscd, mathrsfs}
\usepackage[linktocpage=true]{hyperref}
\usepackage{color}
\usepackage[all]{xy}
\usepackage{tikz}

\newtheorem{thm}{Theorem} [section]
\theoremstyle{definition}

\newtheorem{rem}[thm]{Remark}
\theoremstyle{plain}

\newtheorem{lem}[thm]{Lemma}

\numberwithin{equation}{section}

\newcommand{\al}{\alpha}

\newcommand{\del}{\delta}

\newcommand{\ep}{\varepsilon}

\newcommand{\g}{\mathfrak{g}}
\newcommand{\gl}{\mathfrak{gl}}

\newcommand{\ka}{\zeta}

\newcommand{\la}{\lambda}

\newcommand{\osp}{\text{osp}}

\newcommand{\OO}{\mathcal O}

\newcommand{\Z}{\mathbb Z}

\title{Singular vector formulas for Verma Modules of simple Lie Superalgebras}

\author[Thomas Sale]{Thomas Sale}
\address{Department of Mathematics, University of Virginia, Charlottesville, VA 22904}
\email{tws2mb@virginia.edu}


\begin{document}

\maketitle

\begin{abstract}
For a simple Lie superalgebra of type BDFG, we give explicit formulas for singular vectors in a Verma module of highest weight $\lambda - \rho$, which have weight $s_{\gamma}\lambda - \rho$ for certain positive non-isotropic roots $\gamma.$ This implies the existence of a nonzero homomorphism between the corresponding Verma modules.
\end{abstract}

\maketitle


\section{Introduction}
Let $ \g$ be a complex semisimple Lie algebra with a fixed simple system and $M(\la)$ the Verma module of highest weight $\la - \rho,$ where $\rho$ is the Weyl vector. Denote by $s_{\gamma}$ reflection about the positive root $\gamma.$ Suppose that $s_{\gamma}\la \leq \la.$ Then, Verma's theorem (cf. \cite{V}) states that  there is a nonzero homomorphism $M(s_{\gamma}\la) \rightarrow M(\la)$. Since the fundamental work of Verma (cf. \cite{V}) and Bernstein, Gelfand and Gelfand (cf. \cite{BGG}), Verma modules have played a prominent role in the theory of Lie algebras. 

A generalization of Verma's theorem to symmetrizable Kac-Moody algebras was achieved by Kac and Kazhdan in \cite{KK}. As an anonymous referee pointed out to the author, this adapts readily to the superalgebra setting, giving the result for basic Lie superalgebras as well, when the reflection in question is associated to a nonisotropic root.  In particular, the existence of a Shapovalov determinant formula (cf. \cite[Theorem 1]{KK} and \cite[Theorem 10.2.5]{M}) implies the result on a Zariski dense subset of the weights in question, and a standard density argument (cf. \cite[Theorem 4.2b]{KK} and \cite[Lemma 10]{BGG}) completes the proof.

We introduce some notations and definitions. Take $ \g = \mathfrak{n}^{-} \oplus \mathfrak{h} \oplus \mathfrak{n}^{+}$ to be a basic Lie superalgebra over $\mathbb{C}$ with a fixed triangular decomposition,  such that the corresponding simple system contains at most one isotropic odd root. We choose to consider $\gl(m|n)$ in type A. Denote the Weyl group of $\g$ by $W$ and denote the subgroup generated by reflections about nonisotropic simple roots by $W'.$  Let $\gamma$ be a positive nonisotropic root in the $W'$-orbit of a positive simple root, and let $M(\la)$ be the Verma module of highest weight $\la - \rho,$ where $\rho$ is the Weyl vector.  Recall that a singular vector is a nonzero weight vector, $v \in M(\la)$ such that $xv =0$ for all $x \in \mathfrak{n}^{+}$.  The existence of such a vector, $v$, implies the existence of a nonzero homomorphsim $M(\mu) \rightarrow M(\la)$, where $\mu -\rho$ is the weight of  $v$ (c.f. \cite[\S 2.3.2]{CW2}).

An alternate proof of Verma's theorem was developed by Shapovalov (cf. \cite{S}), which uses the fact that there is an explicit formula for singular vectors associated to simple reflections and that every positive root is in the Weyl group orbit of a simple root. Musson \cite[Corollary 9.2.7]{M}  has adapted this proof to show that there exists a nonzero homomorphism $M(s_{\gamma}\la) \rightarrow M(\la)$, when $\langle \la,h_{\gamma} \rangle = N \in \mathbb{Z}_{>0},$ where $N$ is odd if $\gamma$ is an odd root.  This gives a complete solution for $\gl(m|n)$, as $W = W'$ and every positive nonisotropic root is in the $W$-orbit of an even simple root.  Furthermore, for $\la -\rho$ typical, Musson \cite[Corollary 9.3.6]{M} extended the above result to any positive $\gamma$ that is neither isotropic nor twice another root.

In types other than A and for $\la$ general (i.e. no typicality conditions), there are positive nonisotropic roots to which Musson's method does not apply. Cheng and Wang (see \cite[Lemma 2.5]{CW1}) completed the solution in type $D(2|1; \zeta)$, by developing an explicit formula for a singular vector corresponding to the only positive nonisotropic root not in the $W'$-orbit of a simple root. This was inspired by a prior construction in \cite{KW} in the setting of affine superalgebras.  The success of this case led Cheng and Wang to conjecture (see \cite[Introduction]{CW1}) that a similar construction would work for other basic Lie superalgebras. 

The goal of this paper is to present simple explicit formulas for singular vectors, corresponding to a positive nonisotropic roots not in the $W'$-orbit of a simple root in the following situations: two for $\osp(2n+1|2m)$, $n \geq 1$, one for the standard simple system and one for the opposite simple system, two for $\osp(2m|2m)$, $n\geq 2,$ one for the standard simple system and one for the opposite simple system, and one each for $F(3|1)$ and $G(3)$, for their standard simple systems. For example, in the case of $\osp(2n|2m),$  $n\geq 2,$ standard simple system, we construct a vector of weight $s_{2\del_m}\la - \rho$ in $M(\la)$, where $\la - \rho$ need not be typical.

It is not obvious that these candidate singular vectors are nonzero, as their formulas are not written in a PBW basis of $U(\mathfrak{n}^-)$, and a significant part of the paper is spent proving this fact.  In each case, the proof amounts to showing that if we write the vector in a particular PBW basis, there is a basis element that can be seen to occur with nonzero coefficient in the vector.  This allows us to avoid finding a closed formula for the vector in the standard PBW basis, as Cheng and Wang did in \cite[Lemma 2.5]{CW1}, which would be difficult in the cases dealt with here.  Notably, unlike in the classical case, the homomorphisms determined by the singular vectors associated to even reflections are not necessarily embeddings (see Remark \ref{4.3}).

In the $F(3|1)$ and $G(3)$ cases, these singular vectors correspond to the only positive nonisotropic roots that are not twice another root and not in the $W'$-orbit of a simple root, so that completes these cases.  For the other cases, we adapt the proof of \cite[Theorem 9.2.6]{M} to obtain the result for all roots in the $W'$-orbit of the root correponding to the singular vector.  This completes the case of $\osp(2n|2m)$, $n \geq 2$ with the opposite simple system. In the other three cases, there remain $W'$-orbits of positive even roots to which we were unable to apply our methods.

The paper is organized as follows.  In \S \ref{2}, we introduce the basic terminology, notations and concepts that will be used throughout the paper, as well as the formulas for the candidate singular vectors.  In \S \ref{S3}, we prove that the candidate singular vectors are, in fact, singular.  In \S \ref{S4}, we adapt the proof of \cite[Theorem 9.2.6]{M} to obtain the desired homomorphisms for all roots in the $W'$-orbit of the root corresponding to the singular vector.
\vspace{.2in}

 \section{The Preliminaries and Formulas for Singular Vectors}\label{2}
In this section, we set up notation and introduce some basic ideas and constructions,  (cf. \cite{CW2}).  We also write down formulas for the vectors that we will prove to be singular in \S \ref{S3}.

Let $\g$ be a basic Lie superalgebra over $\mathbb{C}$ with simple system $\Pi$ and triangular decomposition
\[ \g = \mathfrak{n}^{-} \oplus \mathfrak{h} \oplus \mathfrak{n}^{+}, \]
where $\mathfrak{n}^{\pm}$ can be expressed as a direct sum of root spaces:  $\mathfrak{n}^{\pm} = \bigoplus_{\al \in \Phi^{\pm}} \g_{\al}.$
Here, $\Phi^+$ (resp. $\Phi^-$) are the positive roots (resp. negative roots).
We also write $\mathfrak{b}^{\pm} = \mathfrak{n}^{\pm} \oplus \mathfrak{h}.$
We have a decomposition of the positive roots into even and odd roots respectively: $\Phi^+ = \Phi_{0}^+ \cup \Phi_{1}^+. $  Write $W$ for the Weyl group of $\g$ and $W'$ for the subgroup of $W$ generated by reflections about nonisotropic simple roots.

Let $\langle -, - \rangle$ be the standard bilinear pairing on $\mathfrak{h}^* \times \mathfrak{h},$ and let $(-,-)$ be a non-degenerate, $W$-invariant bilinear pairing on  $\mathfrak{h}$ (cf. \cite[Theorem 1.18]{CW2}). For a non-isotropic root $\beta,$ define $h_{\beta}$ to be the coroot of $\beta,$ i.e., the unique element of $\mathfrak{h}$ such that $\langle \mu, h_{\beta} \rangle = \dfrac{2(\mu,\beta)}{(\beta,\beta)}$ for all $\mu \in \mathfrak{h}^*$.

For each $\al \in \Phi^+,$ fix root vectors $e_{\al} \in \g_{\al}$ and $f_{\al} = e_{-\al} \in \g_{-\al}.$
Let $\rho$ be the Weyl vector, which is given by
\[ \rho := \frac{1}{2}\sum_{\al \in \Phi_{0}^+}\al - \frac{1}{2}\sum_{\beta \in \Phi_{1}^+}\beta. \]

We will be concerned with a particular positive nonisotropic root $\gamma.$ Let $\la \in \mathfrak{h}^*$ be any weight such that 
\begin{equation}\label{N}
N := \langle \la, h_{\gamma} \rangle \in \Z_{>0},
\end{equation}
where
\begin{equation}\label{M}
\text{$N = 2M +1$ is odd if $\gamma$ is a nonisotropic odd root.}
\end{equation}
This arises as an integrality condition for the rank 1 subalgebra, $\osp(1|2)$ associated to odd, nonisotropic $\gamma.$
Denote by $M(\la)$ the Verma module of $\g$ of highest weight $\la - \rho$ with highest weight vector $v^+.$ It is defined as
\[ M(\la) := U(\g) \otimes_{U(\mathfrak{b}^+)} \mathbb{C}_{\la - \rho}, \]
where $U(\g)$ is the universal enveloping algebra of $\g$ and $ \mathbb{C}_{\la - \rho} $ is the 1 dimensional $\mathfrak{b}^+$-module upon which $\mathfrak{n}^+$ acts as $0$ and $\mathfrak{h}$ acts by the weight $\la - \rho.$ As in \cite[\S2.3.2]{CW2}, we define a singular vector $v$ of $M(\la)$ to be a nonzero weight vector such that $xv = 0$ for all $x \in \mathfrak{n}^{+}.$

Below we list the relevant Lie superalgebras and their simple systems along with the corresponding Dynkin diagrams, $\Phi_{0}^+$, $\Phi_{1}^+,$  $\gamma,$ and $\rho.$ We also introduce a vector, $u$, which we will show is singular of weight $\la - \rho - N\gamma$ (see \eqref{N}) in the next section.

For $\Phi_{0/1}^+$ in \S \ref{2.1} - \ref{2.4} below, we have that 
 \[ 1 \leq i < j \leq m, 1\leq p \leq m, 1 \leq k < l \leq n, 1 \leq q \leq n. \]

\subsection{$\g = \osp(2n+1|2m)$, $n \geq 1$, I}\label{2.1}
\begin{align*}
&\Pi = \{ \del_i - \del_{i+1}, \del_m -\ep_1, \ep_j -\ep_{j+1}, \ep_n|1 \leq i \leq m-1, 1 \leq j \leq n-1 \}, \\
&\Phi_{0}^+ = \{ \del_i \pm \del_j, 2\del_p, \ep_k \pm \ep_l, \ep_q  \}, \Phi_{1}^+ = \{\del_p, \del_p \pm \ep_q \}, \\
&\gamma = \del_m, \\
&\rho = \sum_{i=1}^{m}(m - n - i + \frac{1}{2})\del_i + \sum_{j=1}^{n}(n - j +\frac{1}{2})\ep_j. \\
\end{align*}

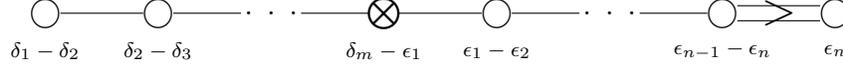
\begin{center}
\begin{tikzpicture}
\node at (0,0) {$\bigcirc$};
\draw (.2,0)--(1.3,0);
\node at (1.5,0) {$\bigcirc$};
\draw (1.7,0)--(2.6,0);
\node at (3,0) {. . .};
\draw (3.4,0)--(4.3,0);
\node at (4.5,0) {$\bigotimes$};
\draw (4.7,0)--(5.8,0);
\node at (6,0) {$\bigcirc$};
\draw (6.2,0)--(7.1,0);
\node at (7.5,0) {. . .};
\draw (7.9,0)--(8.8,0);
\node at (9,0) {$\bigcirc$};
\draw (9.2,-.1)--(10.3,-.1);
\node at (9.75,0) {\Large $>$};
\draw (9.2,.1)--(10.3,.1);
\node at (10.5,0) {$\bigcirc$};
\node at (0,-.5) {\tiny $\delta_1 -\delta_2$};
\node at (1.5,-.5) {\tiny $\delta_2 -\delta_3$};
\node at (4.5,-.5) {\tiny $\delta_m -\epsilon_1$};
\node at (6,-.5) {\tiny $\epsilon_1 -\epsilon_2$};
\node at (9,-.5) {\tiny $\epsilon_{n-1} -\epsilon_n$};
\node at (10.5,-.5) {\tiny $\epsilon_n$};
\end{tikzpicture}
\end{center}

We define a vector
\begin{equation}\label{f1}
u := \prod_{i=1}^n (e_{\delta_m-\ep_i}e_{\delta_m+\ep_i}) f_{\delta_m}^{2M+2n+1} v^+,
\end{equation}
which we will later prove is singular of weight $\la -\rho -(2M+1)\del_m$ (see \eqref{N}).   Recall that $u$ is a singular vector if it is a weight vector and $e_{\al}u = 0$ for all $\al \in \Phi^+$ (cf. \cite[\S 2.3.2]{CW2}). 

Moreover, we are grateful to an anonymous expert for pointing out a mistake in an earlier version of \eqref{f1} above and \eqref{f6} below.  The mistake was caused by the erroneous assumption that $N = 2M$ is even, which has been replaced by the assumption \eqref{M} in these cases.

\subsection{$\g = \osp(2n+1|2m),$ $n \geq 1,$ II}\label{2.2}
\begin{align*}
&\Pi = \{\ep_j - \ep_{j+1}, \ep_n -\del_1, \del_i -\del_{i+1}, \del_m|1 \leq i \leq m-1, 1 \leq j \leq n-1\}, \\
&\Phi_{0}^+ = \{ \del_i \pm \del_j, 2\del_p, \ep_k \pm \ep_l, \ep_q \}, \Phi_{1}^+ = \{\del_p, \ep_q \pm \del_p \}, \\
&\gamma = \ep_n, \\
&\rho = \sum_{i=1}^{m}(m - i + \frac{1}{2})\del_i + \sum_{j=1}^{n}(n - m - j +\frac{1}{2})\ep_j.
\end{align*}

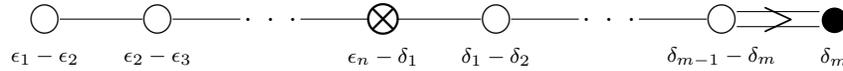
\begin{center}
\begin{tikzpicture}
\node at (0,0) {$\bigcirc$};
\draw (.2,0)--(1.3,0);
\node at (1.5,0) {$\bigcirc$};
\draw (1.7,0)--(2.6,0);
\node at (3,0) {. . .};
\draw (3.4,0)--(4.3,0);
\node at (4.5,0) {$\bigotimes$};
\draw (4.7,0)--(5.8,0);
\node at (6,0) {$\bigcirc$};
\draw (6.2,0)--(7.1,0);
\node at (7.5,0) {. . .};
\draw (7.9,0)--(8.8,0);
\node at (9,0) {$\bigcirc$};
\draw (9.2,-.1)--(10.3,-.1);
\node at (9.75,0) {\Large $>$};
\draw (9.2,.1)--(10.3,.1);
\node at (10.5,-.04) {\Huge $\bullet$};
\node at (0,-.5) {\tiny $\epsilon_1 -\epsilon_2$};
\node at (1.5,-.5) {\tiny $\epsilon_2 -\epsilon_3$};
\node at (4.5,-.5) {\tiny $\epsilon_n -\delta_1$};
\node at (6,-.5) {\tiny $\delta_1 -\delta_2$};
\node at (9,-.5) {\tiny $\delta_{m-1} -\delta_m$};
\node at (10.5,-.5) {\tiny $\delta_m$};
\end{tikzpicture}
\end{center}

We define a vector
\begin{equation}\label{f2}
u := \prod_{i=1}^m (e_{\ep_n-\del_{i}}e_{\ep_n+\del_i}) f_{\ep_n}^{N+2m} v^+,
\end{equation}
which we will later prove is singular of weight $\la -\rho -N\ep_n$ (see \eqref{N}).

We will need the following commutation formulas for root vectors in $\g = \osp(2n+1|2m)$ from \cite[Section 1.2.4]{CW2} for the \S \ref{2.2} case of the proof of Theorem \ref{3} and for Remark \ref{4.3}.
\[
[e_{\ep_n + \del_m}, f_{\ep_n}] = -e_{\del_m}, [e_{\ep_n - \del_m}, f_{\ep_n}] = -e_{-\del_m}, [e_{\del_m}, f_{\ep_n}] = e_{\del_m - \ep_n},
\]
\[
[e_{-\del_m}, f_{\ep_n}] = e_{-\del_m - \ep_n}, [e_{\ep_n - \del_m}, e_{-\ep_n + \del_m}] = \frac{1}{2}(h_{\del_m} + h_{\ep_n}), [e_{\ep_n - \del_m}, e_{\del_m}] = -e_{\ep_n},
\]
\[
[e_{\del_m}, e_{-\del_m}] = \frac{1}{2}h_{\del_m}, [e_{\del_m}, e_{-\ep_n - \del_m}] = -f_{\ep_n}, [e_{-\ep_n + \del_m}, e_{-\del_m}] = f_{\ep_n}.
\]

\subsection{$\g = \osp(2n|2m)$, $n \geq 2$, I}\label{2.3}
\begin{align*}
&\Pi = \{\del_i - \del_{i+1}, \del_m -\ep_1, \ep_j -\ep_{j+1}, \ep_{n-1}+\ep_n|1 \leq i \leq m-1, 1 \leq j \leq n-1\}, \\
&\Phi_{0}^+ = \{ \del_i \pm \del_j, 2\del_p, \ep_k \pm \ep_l \}, \Phi_{1}^+ = \{\del_p \pm \ep_q \}, \\
&\gamma = 2\del_m, \\
&\rho = \sum_{i=1}^{m}(m - n - i + 1)\del_i + \sum_{j=1}^{n}(n - j)\ep_j.
\end{align*}

\begin{center}
\begin{tikzpicture}
\node at (0,0) {$\bigcirc$};
\draw (.2,0)--(1.3,0);
\node at (1.5,0) {$\bigcirc$};
\draw (1.7,0)--(2.6,0);
\node at (3,0) {. . .};
\draw (3.4,0)--(4.3,0);
\node at (4.5,0) {$\bigotimes$};
\draw (4.7,0)--(5.8,0);
\node at (6,0) {$\bigcirc$};
\draw (6.2,0)--(7.1,0);
\node at (7.5,0) {. . .};
\draw (7.9,0)--(8.8,0);
\node at (9,0) {$\bigcirc$};
\draw (9.15, .15)--(9.85,.85);
\node at (10,1) {$\bigcirc$};
\draw (9.15,-.15)--(9.85,-.85);
\node at (10,-1) {$\bigcirc$};
\node at (0,-.5) {\tiny $\delta_1 -\delta_2$};
\node at (1.5,-.5) {\tiny $\delta_2 -\delta_3$};
\node at (4.5,-.5) {\tiny $\delta_m -\epsilon_1$};
\node at (6,-.5) {\tiny $\epsilon_1 -\epsilon_2$};
\node at (8.5,-.5) {\tiny $\epsilon_{n-2} -\epsilon_{n-1}$};
\node at (11, 1) {\tiny $\epsilon_{n-1} - \epsilon_n$};
\node at (11, -1) {\tiny $\epsilon_{n-1} + \epsilon_n$};
\end{tikzpicture}
\end{center}

We define a vector
\begin{equation}\label{f3}
u := \prod_{i=1}^n (e_{\delta_m-\ep_i}e_{\delta_m+\ep_i}) f_{2\delta_m}^{N+n} v^+,
\end{equation}
which we will later prove is singular of weight $\la -\rho -2N\del_m.$

\subsection{$\g = \osp(2n|2m),$ $n \geq 2,$ II}\label{2.4}
\begin{align*}
&\Pi = \{\ep_j - \ep_{j+1}, \ep_n -\del_1, \del_i -\del_{i+1}, 2\del_m|1 \leq i \leq m-1, 1 \leq j \leq n-1\}, \\
&\Phi_{0}^+ = \{ \del_i \pm \del_j, 2\del_p, \ep_k \pm \ep_l, \}, \Phi_{1}^+ = \{\ep_q \pm \del_p \}, \\
&\gamma = \ep_{n-1} + \ep_n, \\
&\rho = \sum_{i=1}^{m}(m - i + 1)\del_i + \sum_{j=1}^{n}(n - m - j)\ep_j.
\end{align*}

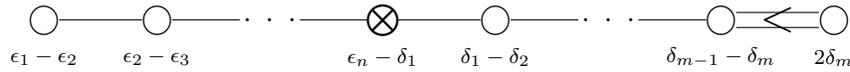
\begin{center}
\begin{tikzpicture}
\node at (0,0) {$\bigcirc$};
\draw (.2,0)--(1.3,0);
\node at (1.5,0) {$\bigcirc$};
\draw (1.7,0)--(2.6,0);
\node at (3,0) {. . .};
\draw (3.4,0)--(4.3,0);
\node at (4.5,0) {$\bigotimes$};
\draw (4.7,0)--(5.8,0);
\node at (6,0) {$\bigcirc$};
\draw (6.2,0)--(7.1,0);
\node at (7.5,0) {. . .};
\draw (7.9,0)--(8.8,0);
\node at (9,0) {$\bigcirc$};
\draw (9.2,-.1)--(10.3,-.1);
\node at (9.75,0) {\Large $<$};
\draw (9.2,.1)--(10.3,.1);
\node at (10.5,0) {$\bigcirc$};
\node at (0,-.5) {\tiny $\epsilon_1 -\epsilon_2$};
\node at (1.5,-.5) {\tiny $\epsilon_2 -\epsilon_3$};
\node at (4.5,-.5) {\tiny $\epsilon_n -\delta_1$};
\node at (6,-.5) {\tiny $\delta_1 -\delta_2$};
\node at (9,-.5) {\tiny $\delta_{m-1} -\delta_m$};
\node at (10.5,-.5) {\tiny $2\delta_m$};
\end{tikzpicture}
\end{center}

We define a vector
\begin{equation}\label{f4}
u := \prod_{i=1}^m (e_{\ep_{n}-\del_{i}}e_{\ep_n+\del_i})\prod_{i=1}^m (e_{\ep_{n-1}-\del_{i}}e_{\ep_{n-1}+\del_i}) f_{\ep_{n-1}+\ep_n}^{N+2m} v^+,
\end{equation}
which we will later prove is singular of weight $\la - \rho - N(\ep_{n-1} + \ep_n).$

\subsection{$\g = F(3|1)$}\label{2.5}
\begin{align*}
&\Pi = \{ \ep_1 - \ep_2, \ep_2 - \ep_3, \ep_3, \frac{1}{2}(\del - \ep_1 - \ep_2 - \ep_3) \}, \\
&\Phi_{0}^+ = \{ \del, \ep_1, \ep_2, \ep_3, \ep_i \pm \ep_j|1\leq i < j \leq 3 \}, \Phi_{1}^+ = \{ \frac{1}{2}(\del \pm \ep_1  \pm \ep_2 \pm \ep_3) \}, \\
&\gamma = \del, \\
&\rho = \frac{1}{2}(-3\del + 5\ep_1 + 3\ep_2 + \ep_3).
\end{align*}
We introduce some notation for $\g = F(3|1).$  An odd root is denoted by the ordered tuple of signs appearing in it. For example $\frac{1}{2}(\del - \ep_1 - \ep_2 +\ep_3)$ is denoted by $+--+$.

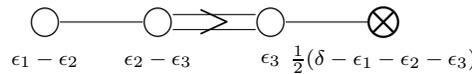
\begin{center}
\begin{tikzpicture}
\node at (0,0) {$\bigcirc$};
\draw (.2,0)--(1.3,0);
\node at (1.5,0) {$\bigcirc$};
\draw (1.7,.1)--(2.8,.1);
\draw (1.7,-.1)--(2.8,-.1);
\node at (2.25,0) {\Large $>$};
\node at (3,0) {$\bigcirc$};
\draw (3.2,0)--(4.3,0);
\node at (4.5,0) {$\bigotimes$};
\node at (0,-.5) {\tiny $\epsilon_1 - \epsilon_2$};
\node at (1.5, -.5) {\tiny $\epsilon_2 - \epsilon_3$};
\node at (3, -.5) {\tiny $\epsilon_3$};
\node at (4.5, -.5) {\tiny $\frac{1}{2}(\delta - \epsilon_1 - \epsilon_2 - \epsilon_3)$};
\end{tikzpicture}
\end{center}

We define a vector
\begin{equation}\label{f5}
u := e_{+---}e_{+--+}e_{+-+-}e_{+-++}e_{++--}e_{++-+}e_{+++-}e_{++++} f_{\delta}^{N+4} v^+,
\end{equation}
which we will later prove is singular of weight $\la -\rho -N\del.$

\subsection{$\g = G(3)$}\label{2.6}
\begin{align*} 
&\Pi = \{ \ep_{2} - \ep_{1}, \ep_1 , \del + \ep_3 \}, \ep_1 + \ep_2 + \ep_3 = 0, \\
&\Phi_{0}^+ = \{  2\del, \ep_1, \ep_2, -\ep_3, \ep_2 - \ep_1, \ep_1 - \ep_3, \ep_2 - \ep_3\}, \Phi_{1}^+ = \{\ \del, \del \pm \ep_i| 1 \leq i \leq 3\}, \\
&\gamma = \del, \\
&\rho = -\frac{5}{2} \del + 2\ep_1 + 3\ep_2.
\end{align*}

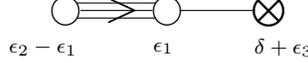
\begin{center}
\begin{tikzpicture}
\node at (0,0) {$\bigcirc$};
\draw (0.2,0)--(1.155,0);
\draw (0.15,0.1)--(1.2,0.1);
\draw (0.15,-0.1)--(1.2,-0.1);
\node at (1.35,0) {$\bigcirc$};
\node at (0.75,0) {\Large $>$};
\draw (1.52,0)--(2.52,0);
\node at (2.7,0) {$\bigotimes$};
\node at (-0.3,-.5) {\tiny $\epsilon_2-\epsilon_1$};
\node at (1.3,-.5) {\tiny $\epsilon_1$};
\node at (2.9,-.5) {\tiny $\delta+\epsilon_3$};
\end{tikzpicture}
\end{center}

We define a vector
\begin{equation}\label{f6}
u := e_{\del - \ep_1}e_{\del + \ep_1}e_{\del - \ep_2}e_{\del + \ep_2}e_{\del - \ep_3}e_{\del + \ep_3} f_{\delta}^{2M+7} v^+,
\end{equation}
which we will later prove is singular of weight $\la -\rho -(2M+1)\del.$

\section{The vectors $u$ are singular}\label{S3}

Recall the positive nonisotropic roots $\gamma$ and vectors $u$ defined in \S \ref{2.1} - \ref{2.6}.  In this section, we will prove that the vectors $u$ are singular, which implies the existence of a nonzero homomorphism $M(s_{\gamma}\la) \rightarrow M(\la).$

\begin{lem}\label{L3.1}
Changing the order of the product of positive odd root vectors in the formulas for $u$ given in \S \ref{2.1} - \ref{2.6} changes $u$ by a factor of $\pm1.$
\end{lem}

\begin{proof}
In the cases of \S\ref{2.1} and \S\ref{2.3}, among the isotropic root vectors in the formula for $u$, the only one failing to anti-commute with $e_{\del_m \pm \ep_k}$ is $e_{\del_m \mp \ep_k}$.  Instead $[e_{\del_m \pm \ep_k},e_{\del_m \mp \ep_k}]$ is a multiple of $e_{2\del_m}.$ Let $u'$ be any vector obtained by rearranging the terms of the product.  Then, in the case of \S\ref{2.3}, $u'$ is equal to $\pm u$ plus a linear combination of terms of the form
\[ v := \prod_{i=1}^{k-1} (e_{\delta_m-\ep_i}e_{\delta_m+\ep_i})\prod_{i=k+1}^n (e_{\delta_m-\ep_i}e_{\delta_m+\ep_i}) e_{2\del_m}f_{2\delta_m}^{N+n} v^+, \]
and we have by $\mathfrak{sl}_2$ relations,
\begin{align}
\begin{split}
e_{2\del_m}f_{2\del_m}^{N+n}v^+
&= f_{2\del_m}^{N+n}e_{2\del_m}v^+ + (N+n)f_{2\del_m}^{N+n-1}(h_{2\del_m} - (N+n-1))v^+ \\
&= (N+n)f_{2\del_m}^{N+n-1}(\langle \la - \rho, h_{2\del_m} \rangle - (N +n -1))v^+ = 0. 
\end{split}
\end{align}
Hence, $v = 0$ and $u' = \pm u.$ The cases of \S\ref{2.2}, \S\ref{2.4} and \S\ref{2.5} are similar.

In the case of \S\ref{2.1}, $u'$ is equal to $\pm u$ plus a linear combination of terms of the form:
\[ v' := \prod_{i=1}^{k-1} (e_{\delta_m-\ep_i}e_{\delta_m+\ep_i})\prod_{i=k+1}^n (e_{\delta_m-\ep_i}e_{\delta_m+\ep_i}) e_{\del_m}^{2}f_{\delta_m}^{2M+2n+1} v^+. \]

In order to show that $v' = 0$, we must first recall an $\osp(1|2)$ relation from \cite[\S A.4.4]{M}.  Note that $e_{\del_m}$ and $f_{\del_m}$ generate a copy of $\osp(1|2)$ in $\osp(2m+1|2n)$, and suppose $[e_{\del_m}, f_{\del_m}] = \frac{1}{2}h_{\del_m}$. Let $w$ be a highest weight vector in a Verma Module for $\osp(1|2)$.  Then we have the following formula:
\[ e_{\del_m}f_{\del_m}^{2n+1}w = f_{\del_m}^{2n}(\frac{1}{2}h_{\del_m} - n)w.\]

Thus,
\begin{align}
\begin{split}
e_{\del_m}f_{\del_m}^{2M+2n+1}v^+
&= f_{\del_m}^{2M+2n}(\frac{1}{2}h_{\del_m} - (M+n))v^+ \\
&= f_{\del_m}^{2M+2n}(\frac{1}{2}\langle \la - \rho, h_{\del_m} \rangle - (M +n))v^+ = 0. 
\end{split}
\end{align}
Hence, $v' = 0$ and $u' = \pm u.$ The case of \S\ref{2.6} is similar.
\end{proof}
\vspace{.3in}

\begin{lem}
The vectors $u$ defined in \S \ref{2.1} - \ref{2.6} are nonzero.
\end{lem}

\begin{proof}
In each case, we will proceed as follows.  We will fix a basis of $M(\la),$ and describe a sequence of elements of this basis $v_k,$ and corresponding vectors $u_k,$ such that $u_K = v_K,$ where K is the largest index in each case and $u_{l} = u,$ where $l$ is the smallest index.  Then, we will show that the coefficient of $v_k$ in $u_k$ is a nonzero multiple of the coefficient of $v_{k+1}$ in $u_{k+1}$ for each $k$, proving the lemma.

In the case of \S\ref{2.1}, we take the basis to be any PBW basis of $M(\la)$ that contains the following vectors:
\begin{equation}\label{3.3}
e_{-\ep_n}^{d_n}. . .e_{-\ep_k}^{d_k}e_{-\ep_n - \ep_{n-1}}^{a_{n-1,n,-}}e_{\ep_n - \ep_{n-1}}^{a_{n-1,n,+}}. . .e_{-\ep_{k+1} - \ep_{k}}^{a_{k,k+1,-}}e_{\ep_{k+1} - \ep_{k}}^{a_{k,k+1,+}}e_{-\del_m + \ep_n}^{b_n}e_{-\del_m - \ep_n}^{c_n}. . . e_{-\del_m + \ep_k}^{b_k}e_{-\del_m - \ep_k}^{c_k}f_{\del_m}^{s}f_{2\delta_m}^{R} v^+
\end{equation}
 of weight $\la -\rho -(2M-1+2k)\del_m,$ where $a_{i,j,\pm}, d_{r}, R \in \Z_{\geq0}$ and $b_p, c_q, s \in \{0,1\}.$ In \eqref{3.3}, $e_{\pm\ep_j - \ep_{i}}$ can appear for any $i,j$ such that $k \leq i <j \leq n$.  From left to right in \eqref{3.3}, the sum of the indices, i.e $i +j$, of the $e_{\pm\ep_j - \ep_{i}}$ is decreasing.  Beyond that, the particular ordering fixed for these terms does not matter. 

We define the following:
\begin{align*}
u_k :&= \prod_{i=k}^n (e_{\delta_m-\ep_i}e_{\delta_m+\ep_i}) f_{\del}f_{2\delta_m}^{M+n} v^+, \\ 
u_{n+1} &= v_{n+1} := f_{\del}f_{2\delta_m}^{M+n} v^+,\\ 
v_k :&= e_{-\ep_n}e_{\ep_n - \ep_{n-1}}. . .e_{\ep_{k+1} - \ep_k}e_{-\del_m + \ep_k}f_{2\delta_m}^{M+k-1} v^+.
\end{align*}
We make use of the fact that $f_{\del_m}^{2} = Cf_{2\del_m}$, where $C$ is a nonzero scalar, so $u_{1}$ is a nonzero multiple of $u$ in this case.  We suppose that for some $k$, $u_k$ can be written as a linear combination of elements of the form \eqref{3.3}, and that the coefficient of $v_k$ in $u_k$ is nonzero.  Clearly, these statements hold for $k = n + 1.$  As $u_{k-1} = e_{\del_m - \ep_{k-1}}e_{\del_m + \ep_{k-1}}u_k,$ it is easy to check that $v_{k-1}$ can only be obtained from the term $v_{k}$ among those of the form \eqref{3.3} by multiplying on the left by $e_{\del_m - \ep_{k-1}}e_{\del_m + \ep_{k-1}},$ by commuting $e_{\del_m - \ep_{k-1}}$ with $e_{-\del_m + \ep_k}$ and $e_{\del_m + \ep_{k-1}}$ with the $f_{2\del_m}$. The exception is for $k = n+1$, in which case $e_{\del_m - \ep_{n}}$ is commuted with $f_{\del_m}$ and $e_{\del_m + \ep_{n}}$ is commuted with the $f_{2\del_m}.$  Hence, the coefficient of $v_{k-1}$ in $u_{k-1}$ is nonzero.  It is also easy to check that multiplying on the left by $e_{\del_m - \ep_{k-1}}e_{\del_m + \ep_{k-1}}$ preserves the form \eqref{3.3} with $k$ replaced by $k-1$.  This completes the induction and the proof of the lemma in the case of \S\ref{2.1}.

The proof in the case of \S\ref{2.2} is similar to the one in the case of \S\ref{2.1}.  We take the basis to be any PBW basis containing the following vectors:
\begin{equation}\label{3.4}
e_{-\del_m}^{a_{m}}. . .e_{-\del_k}^{a_{k}}e_{-\del_m - \del_{m-1}}^{b_{m-1,m,-}}e_{\del_m - \del_{m-1}}^{b_{m-1,m,+}}. . . e_{-\del_{k+1} - \del_{k}}^{b_{k,k+1,-}}e_{\del_{k+1} - \del_{k}}^{b_{k,k+1,+}}e_{-\ep_n + \del_m}^{c_{m}}e_{-\ep_n - \del_m}^{d_m}. . .e_{-\ep_n + \del_k}^{c_k}e_{-\ep_n - \del_k}^{d_k}f_{\ep_n}^{R}v^+
\end{equation} 
of weight  $\la-\rho - (N+2k -2)\ep_n,$ where the $a$'s , the $b$'s and $R$ are non-negative integers and the $c$'s and $d$'s are 0 or 1. In \eqref{3.4}, $e_{\pm\del_j - \del_i}$ can appear for any $i, j$ such that $k \leq i < j \leq m$.  From left to right in \eqref{3.4}, the sum of the indices, i.e. $i+j$, of $e_{\pm\del_j - \del_i}$  is decreasing.  Other than that, the fixed ordering of these terms in \eqref{3.4} does not matter.
The terms of the form \eqref{3.4} here play the same role as those of the form \eqref{3.3} in the case of \S\ref{2.1}.
We define the following:
\begin{align*}
u_k :&= \prod_{i=k}^m (e_{\ep_n-\del_{i}}e_{\ep_n+\del_i}) f_{\ep_n}^{N+2m} v^+, \\ 
u_{m+1} &= v_{m+1} := f_{\ep_n}^{N+2m} v^+, \\
v_k :&= e_{-\del_m}e_{\del_m - \del_{m-1}}. . .e_{\del_{k+1} - \del_k}e_{-\ep_n + \del_{k}}f_{\ep_{n}}^{N+2k-3}v^+.
\end{align*}
The above vectors play the same role as the vectors of the same name in the proof of the \S \ref{2.1} case. The vector $v_{k-1}$ is obtained from $v_k$ by left multiplication by $e_{\ep_n - \del_{k-1}}e_{\ep_n + \del_{k-1}}$ by commuting $e_{\ep_n + \del_{k-1}}$ with the $f_{\ep_n}$ and $e_{\ep_n -\del_{k-1}}$ with $e_{-\ep_n + \del_k},$ except for $k=m+1$, in which case both $e_{\ep_n + \del_m}$ and $e_{\ep_n - \del_m}$ are commuted with the $f_{\ep_n}.$

Similarly, in the case of \S\ref{2.3}, we take the basis to be a PBW basis of $M(\la)$ that contains the following vectors:
\begin{equation}\label{3.5}
e_{-\ep_n - \ep_{n-1}}^{a_{n-1,n,-}}e_{\ep_n - \ep_{n-1}}^{a_{n-1,n,+}}. . .e_{-\ep_{k+1} - \ep_{k}}^{a_{k,k+1,-}}e_{\ep_{k+1} - \ep_{k}}^{a_{k,k+1,+}}e_{-\del_m + \ep_n}^{b_n}e_{-\del_m - \ep_n}^{c_n}. . . e_{-\del_m + \ep_k}^{b_k}e_{-\del_m - \ep_k}^{c_k}f_{2\delta_m}^{R} v^+
\end{equation}
 of weight $\la -\rho -2(N +k -1)\del_m,$ where $a_{i,j,\pm}, R \in \Z_{\geq0}$ and $b_s, c_t \in \{0,1\}.$ In \eqref{3.5}, $e_{\pm\ep_j - \ep_{i}}$ can appear for any $i,j$ such that $k \leq i <j \leq n$.  From left to right in \eqref{3.5}, the sum of the indices, i.e $i +j$, of the $e_{\pm\ep_j - \ep_{i}}$ is decreasing.

We define:
\begin{align*}
u_k :&= \prod_{i=k}^n (e_{\delta_m-\ep_i}e_{\delta_m+\ep_i}) f_{2\delta_m}^{N+n} v^+, \\ 
u_{n+1} &= v_{n+1} := f_{2\delta_m}^{N+n} v^+,\\ 
v_k :&= e_{\ep_n - \ep_{n-1}}. . .e_{\ep_{k+1} - \ep_k}e_{-\del_m - \ep_n}e_{-\del_m + \ep_k}f_{2\delta_m}^{N+k -2} v^+.
\end{align*}
The $u_k$ and $v_k$ play the same role as before.  Here, $v_{k-1}$ is obtained from the term $v_{k}$ by multiplying on the left by $e_{\del_m - \ep_{k-1}}e_{\del_m + \ep_{k-1}},$ by commuting $e_{\del_m - \ep_{k-1}}$ with $e_{-\del_m + \ep_k}$ and $e_{\del_m + \ep_{k-1}}$ with the $f_{2\del_m}$. The exception is for $k = n+1$, in which case  both $e_{\del_m - \ep_{n}}$ and $e_{\del_m + \ep_{n}}$ are commuted with the $f_{2\del_m}.$

The proof in the case of \S\ref{2.4} is similar to the one in the case of \S\ref{2.1}. We take as a basis any PBW basis containing the following vectors:
\begin{equation}\label{3.6}
\begin{split}
&e_{-2\del_m}^{l_m}. . . e_{-2\del_k}^{l_k}e_{-\del_m - \del_{m-1}}^{a_{m-1,m,-}}e_{\del_m - \del_{m-1}}^{a_{m-1,m,+}}. . . e_{-\del_{k+1} - \del_{k}}^{a_{k,k+1,-}}e_{\del_{k+1} - \del_{k}}^{a_{k,k+1,+}}e_{-\ep_n + \del_m}^{b_m}e_{-\ep_n - \del_m}^{c_m}e_{-\ep_{n-1} + \del_m}^{d_m}e_{-\ep_{n-1} - \del_m}^{g_m}. . . \\ 
\cdot &e_{-\ep_n + \del_k}^{b_k}e_{-\ep_n - \del_k}^{c_k}e_{-\ep_{n-1} + \del_k}^{d_k}e_{-\ep_{n-1} - \del_k}^{g_k}e_{\ep_n - \ep_{n-1}}^{S}f_{\ep_{n-1} + \ep_n}^{R}v^+,
\end{split}
\end{equation} 
of weight $\la - \rho -(N +2(k-1))(\ep_{n-1} + \ep_n)$,
 where the $a$, $l$, $R$, and $S$ are in  $\Z_{\geq 0}$ and the $b,c,d,$ and $g$ are $0$ or $1$. The comments on the indices and ordering of the $e_{\pm\del_j - \del_i}$ terms in \eqref{3.4} apply here as well.
The terms of the form \eqref{3.6} the play the same role as those of the form \eqref{3.3} in the case of \S\ref{2.1}.
As before, we define the following:
\begin{align*}
 u_k :&= \prod_{i=k}^m (e_{\ep_{n}-\del_{i}}e_{\ep_n+\del_i} e_{\ep_{n-1}-\del_{i}}e_{\ep_{n-1}+\del_i}) f_{\ep_{n-1}+\ep_n}^{N+2m} v^+, \\
 u_{m+1} &= v_{m+1} :=  f_{\ep_{n-1}+\ep_n}^{N+2m} v^+, \\
 v_k :&= e_{-2\del_m}e_{\del_m-\del_{m-1}}^{2}. . . e_{\del_{k+1}-\del_k}^{2}e_{-\ep_n + \del_k}e_{-\ep_{n-1} + \del_k}f_{\ep_{n-1} + \ep_n}^{N+2k -3}v^+.
\end{align*}
The above vectors play the same role as the vectors of the same name in the proof of the \S \ref{2.1} case. Here $v_{k-1}$ is obtained from $v_k$ by left multiplication by $e_{\ep_n - \del_{k-1}}e_{\ep_n + \del_{k-1}}e_{\ep_{n-1} -\del_{k-1}}e_{\ep_{n-1} + \del_{k-1}}$  by commuting $e_{\ep_{n-1} +\del_{k-1}}$ and $e_{\ep_{n} + \del_{k-1}}$ with the $f_{\ep_{n-1} + \ep_n}$ and by commuting $e_{\ep_{n-1} -\del_{k-1}}$ and $e_{\ep_{n} - \del_{k-1}}$ with $e_{-\ep_{n-1} + \del_k}$ and $e_{-\ep_{n} + \del_k}$ respectively. The exception is for $k = m+1,$ in which case, $e_{\ep_{n-1} - \del_m},$ $e_{\ep_{n-1} + \del_m}$ and $e_{\ep_n + \del_m}$ are commuted with the $f_{\ep_{n-1} + \ep_n}$ and $e_{\ep_n - \del_m}$ is commuted with $e_{-\ep_n - \del_m}.$

In the case of \S\ref{2.5}, we use the following basis for $M(\la):$
\begin{equation}\label{3.7}
\begin{split}
&e_{-\ep_3}^{a_1}e_{-\ep_2}^{a_2}e_{-\ep_1}^{a_3}e_{-\ep_2 - \ep_3}^{b_1}e_{-\ep_2 + \ep_3}^{b_2}e_{-\ep_1 - \ep_3}^{b_3}e_{-\ep_1 + \ep_3}^{b_4}e_{-\ep_1 - \ep_2}^{b_5}e_{-\ep_1 + \ep_2}^{b_6} \\ 
&\cdot e_{-+++}^{c_1}e_{-++-}^{c_2}e_{-+-+}^{c_3}e_{-+--}^{c_4}e_{--++}^{c_5}e_{--+-}^{c_6}e_{---+}^{c_7}e_{----}^{c_8}f_{\del}^{R}v^+,
\end{split}
\end{equation}
where $a_i$, $b_j$ and $R \in \Z_{\geq 0},$ and $c_{k} \in \{0,1\}.$

Consider the basis element $v_{i}$ in $u_{i}$, where
\begin{align*}
&v_{0} := e_{-\ep_3}e_{\ep_3 - \ep_2}e_{\ep_2 - \ep_1}e_{-+++}e_{-+--} f_{\del}^{N-1} v^+, \\
&v_{1} := e_{\ep_3 - \ep_2}e_{\ep_2 - \ep_1}e_{-+++}e_{-++-}e_{-+--} f_{\del}^{N-1} v^+, \\ 
&v_{2} :=  e_{\ep_2 - \ep_1}e_{-+++}e_{-++-}e_{-+-+}e_{-+--} f_{\del}^{N-1} v^+, \\ 
&v_{3} := e_{-+++}e_{-++-}e_{-+-+}e_{-+--}e_{--++} f_{\del}^{N-1} v^+, \\
&v_{4} :=  e_{-+++}e_{-++-}e_{-+-+}e_{-+--} f_{\del}^{N} v^+, \\
&v_{5} := e_{-+++}e_{-++-}e_{-+-+} f_{\del}^{N+1} v^+, \\
&v_{6} := e_{-+++}e_{-++-} f_{\del}^{N+2} v^+, \\
&v_{7} := e_{-+++} f_{\del}^{N+3} v^+, \\
&v_{8} := f_{\delta}^{N+4} v^+,
\end{align*}
\begin{align*}
&u_{0} := u, \\
&u_{1} := e_{+--+}e_{+-+-}e_{+-++}e_{++--}e_{++-+}e_{+++-}e_{++++} f_{\delta}^{N+4} v^+, \\
&u_{2} := e_{+-+-}e_{+-++}e_{++--}e_{++-+}e_{+++-}e_{++++} f_{\delta}^{N+4} v^+, \\
&u_{3} := e_{+-++}e_{++--}e_{++-+}e_{+++-}e_{++++} f_{\delta}^{N+4} v^+,\\ 
&u_{4} := e_{++--}e_{++-+}e_{+++-}e_{++++} f_{\delta}^{N+4} v^+, \\
&u_{5} := e_{++-+}e_{+++-}e_{++++} f_{\delta}^{N+4} v^+, \\
&u_{6} := e_{+++-}e_{++++} f_{\delta}^{N+4} v^+, \\
&u_{7} := e_{++++} f_{\delta}^{N+4} v^+, \\
&u_{8} := f_{\delta}^{N+4} v^+.
\end{align*}

We call the sum of the exponents of the negative root vectors other than $f_{\del}$ in a basis element its degree.  Multiplying on the left by a positive odd root vector, degree can only increase and the exponent of $f_{\del}$ can only decrease by commuting the positive root vector past the $f_{\del}.$ In such a case, the degree increases by one and the exponent of $f_{\del}$ decreases by one. Moreover, if a basis element, $w'$, with greater exponent of $f_{\del}$ is gotten from one, $w$, with lesser exponent of $f_{\del}$ by left multiplication by a positive odd root vector, then the sum of the degree and the exponent of $f_{\del}$ of $w'$ is less than that of $w$. Hence, for a basis element appearing in one of the $u_i$, the sum of the degree and the exponent of $f_{\del}$ is less than or equal to $N+4.$ Note that this sum is exactly $N + 4$ for the $v_i.$

Thus, the vector $v_{0}$ is obtained from terms of degree 4 with exponent of $f_{\del}$ equal to $N$ or degree 5 with exponent of $f_{\del}$ equal to $N-1$ in $u_{1}$ by applying $e_{+---}.$  Because the term $e_{----}$ does not appear in $v_{0},$ it must have come from terms of degree 5.  But, of the terms other than $f_{\del}$ appearing in $v_{0},$ only $e_{-\ep_3}$ lies in the image of $ad_{e_{+---}},$ so $v_{0}$ could only have come from $v_{1}$, with coefficient of $v_{0}$ in $u$ a nonzero multiple of that of $v_{1}$ in $u_1$.  Using similar arguments, one sees that the  $v_{i}$  term in $u_{i}$ could only have come from the $v_{i+1}$ term in $u_{i+1}$ by applying the appropriate positive odd root vector, and that the coefficient of $v_{i}$ in $u_{i}$ is a nonzero multiple of that of $v_{i+1}$ in $u_{i+1}.$  As the coefficient of $v_{8}$ in $u_{8}$ is 1, the coefficient of $v_{0}$ in $u$ is nonzero, and, consequently, $u \neq 0.$

In the case of \S\ref{2.6}, we consider the following basis for $M(\la)$:
\begin{equation}\label{3.7}
e_{\ep_3}^{a_1}e_{-\ep_2}^{a_2}e_{-\ep_1}^{a_3}e_{\ep_3 - \ep_2}^{b_1}e_{\ep_3 - \ep_1}^{b_2}e_{\ep_1 - \ep_2}^{b_3} e_{-\del + \ep_3}^{c_1}e_{-\del - \ep_3}^{c_2}e_{-\del + \ep_2}^{c_3}e_{-\del - \ep_2}^{c_4}e_{-\del + \ep_1}^{c_5}e_{-\del - \ep_1}^{c_6}f_{\del}^{s} f_{2\delta}^{R} v^+,
\end{equation}
where $a_{i}, b_{j}$ and $R \in \Z_{\geq 0},$ and $c_{k}, s \in \{0, 1\}.$

Consider the following basis elements $v_i$ in $u_i$:
\begin{align*}
&v_0 := e_{-\ep_1}^{2}e_{\ep_1 - \ep_2}e_{-\del - \ep_3}f_{2\del}^{M}v^+, \\
&v_1 := e_{-\ep_1}^{2}e_{-\del - \ep_3}e_{-\del - \ep_2}f_{2\del}^{M}v^+, \\
&v_2 := e_{-\ep_1}e_{-\del +\ep_3}e_{-\del - \ep_3}e_{-\del - \ep_2}f_{2\del}^{M}v^+, \\
&v_3 := e_{-\del + \ep_3}e_{-\del - \ep_3}e_{-\del - \ep_2}f_{\del}f_{2\del}^{M}v^+, \\
&v_4 := e_{-\del + \ep_3}e_{-\del -\ep_3}f_{\del}f_{2\del}^{M+1}v^+, \\
&v_5 := e_{-\del - \ep_3}f_{\del}f_{2\del}^{M+2}v^+, \\
&v_6 := f_{\del}f_{2\del}^{M+3}v^+,
\end{align*}
\begin{align*}
&u_0 := e_{\del + \ep_1}e_{\del + \ep_2}e_{\del - \ep_1}e_{\del - \ep_2}e_{\del + \ep_3}e_{\del -\ep_3} f_{\del}f_{2\delta}^{M+3} v^+, \\
&u_1 := e_{\del + \ep_2}e_{\del - \ep_1}e_{\del - \ep_2}e_{\del + \ep_3}e_{\del -\ep_3} f_{\del}f_{2\delta}^{M+3} v^+ , \\
&u_2 := e_{\del - \ep_1}e_{\del - \ep_2}e_{\del + \ep_3}e_{\del -\ep_3} f_{\del}f_{2\delta}^{M+3} v^+ , \\
&u_3 := e_{\del - \ep_2}e_{\del + \ep_3}e_{\del -\ep_3} f_{\del}f_{2\delta}^{M+3} v^+ , \\
&u_4 := e_{\del + \ep_3}e_{\del -\ep_3} f_{\del}f_{2\delta}^{M+3} v^+ , \\
&u_5 := e_{\del -\ep_3} f_{\del}f_{2\delta}^{M+3} v^+ , \\
&u_6 := f_{\del}f_{2\delta}^{M+3} v^+ .
\end{align*}

Note that $u_0$ is a nonzero multiple of $u$. We show that the coefficient of the basis element $v_0$ in $u_0$ is nonzero, and hence that $u$ is nonzero.  This term is obtained by applying $e_{\del + \ep_1}$ to $u_1$ written in the basis and reordering terms as necessary. Similarly to the proof of the case of \S \ref{2.5}, we say that the degree of a basis element is the sum of the exponents of the negative root vectors other than $f_{2\del}$ appearing in it.  Multiplying on the left by a positive odd root vector other than $e_{\del}$, degree can only increase and the exponent of $f_{2\del}$ can only decrease by commuting the positive root vector past the $f_{2\del}.$ This increases degree by one and also reduces the exponent of $f_{2\del}$ by one.  Moreover, if a basis element, $w'$, with greater exponent of $f_{2\del}$ is gotten from one, $w$, with lesser exponent of $f_{2\del}$ by left multiplication by a positive odd root vector other than $e_{\del}$, then the sum of the degree and the exponent of $f_{2\del}$ of $w'$ is less than that of $w$. Hence, the sum of the degree and the exponent of $f_{2\del}$ is less than or equal to $M +4$ for any basis element appearing in one of the $u_i.$ Note that this sum is exactly $M+4$ for the $v_i.$

Since $v_0$ does not contain the term $e_{-\del + \ep_1},$ $v_0$ must have come from terms of degree 4 and exponent of $f_{2\del}$ equal to $M$ in $u_1$ by applying $e_{\del + \ep_1}.$  Of the root vectors other than $f_{2\del}$ appearing in $v_0$, only $e_{\ep_1 - \ep_2}$ lies in the image of $ad_{e_{\del + \ep_1}},$ so the coefficient of $v_0$ in $u_0$ is a non-zero multiple of the coefficient of $v_1$ in $u_1$. With similar reasoning, we see that the coeffcient of $v_i$ in $u_i$ is a nonzero multiple of that $v_{i+1}$ in $u_{i+1}$. As the coefficient of $v_{6}$ in $u_{6}$ is 1, the coefficient of $v_0$ in $u_0$ is non-zero, and hence $u \neq 0.$
\end{proof}

Now that we know $u \neq 0,$ we are in the position to prove that the vectors $u$ are singular.

\begin{thm}\label{3}
The vectors $u$ defined in \S \ref{2.1} - \ref{2.6} are singular.
\end{thm}

\begin{proof}We begin with the case of \S\ref{2.1}.  In all the cases, we check that $e_{\al}u =0,$ where $\al \in \Pi.$ The $c_i$, $d_k,$ and $g_k$  below denote scalars.
First, note that  $\la -\rho -2N\del_m + \del_{i} - \del_{i+1}$ is not a lower weight than $\la - \rho.$  
Hence, $e_{\del_i - \del_{i+1}}u = 0.$  
Also, it is clear that $e_{\del_m - \ep_1}u = 0,$ as $e_{\del_m - \ep_1}^{2} = 0$. We also have that 
\begin{align}\label{3.8}
\begin{split} 
e_{\ep_{k} - \ep_{k+1}}u = &c_1 \prod_{i=1}^{k-1} (e_{\delta_m-\ep_i}e_{\delta_m+\ep_i})e_{\del_m - \ep_k}e_{\del_m + \ep_k}^{2}e_{\del_m - \ep_{k+1}}\prod_{i=k+2}^{n} (e_{\delta_m-\ep_i}e_{\delta_m+\ep_i}) f_{\delta_m}^{2M+n+1} v^+ \\
&+  c_2 \prod_{i=1}^{k-1} (e_{\delta_m-\ep_i}e_{\delta_m+\ep_i})e_{\del_m + \ep_k}e_{\del_m - \ep_{k+1}}^{2}e_{\del_m + \ep_{k+1}}\prod_{i=k+2}^{n} (e_{\delta_m-\ep_i}e_{\delta_m+\ep_i}) f_{\delta_m}^{2M+n+1} v^+ = 0.
\end{split}
\end{align} 
The particular values of the scalars, $c_1$ and $c_2$, in \eqref{3.8} do not matter as each term is 0.  This phenomenon will be observed throughout the proof.  Finally, by $\osp(1|2)$ relations, we have
\begin{align*}
\begin{split}
 e_{\ep_n}u &= c_3 \prod_{i=1}^{n-1}(e_{\delta_m-\ep_i}e_{\delta_m+\ep_i})e_{\delta_m + \ep_n}e_{\delta_m} f_{\del_m}f_{2\delta_m}^{M+n} v^+ + c_4\prod_{i=1}^{n}(e_{\delta_m-\ep_i}e_{\delta_m+\ep_i})e_{-\delta_m + \ep_n} f_{2\delta_m}^{M+n} v^+ \\
&= 0 + c_5\prod_{i=1}^{n-1}(e_{\delta_m-\ep_i}e_{\delta_m+\ep_i})e_{\delta_m - \ep_n}e_{-\delta_m + \ep_n}^{2} f_{2\delta_m}^{M+n-1} v^+ \\
&= 0.
\end{split}
\end{align*}

This case of \S\ref{2.2} is similar to that of \S\ref{2.1} with the roles of the $\del_i$ and $\ep_k$ reversed.  The one exception is that we will have to use a different method for $\del_m$ from the one we used for $\ep_n$. Here, we use the commutation formulas of \S \ref{2.2}. A tedious computation shows that 
\begin{align}
\begin{split}
w :&= e_{\del_m}e_{\ep_n -\del_m}e_{\ep_n + \del_m}f_{\ep_n}^{N +2m}v^+ \\ 
&= -A(N - 2m -2)[a - b + (N -2m -3)/2]e_{-\ep_n + \del_m}f_{\ep_n}^{N - 2m - 3}v^+,
\end{split}
\end{align}
where $A:= \frac{1}{2}(N + 2m)(N + 2m -1), 
b:= \langle \la - \rho -(N -2m -3)\ep_n, \frac{1}{2}h_{\del_m} \rangle,$ and 
\begin{align}
\begin{split}
a :&= \langle \la - \rho - (N + 2m -2)\ep_n, \frac{1}{2}(h_{\ep_n} +h_{ \del_m}) \rangle\\ 
&= \frac{1}{2}[\langle \la -\rho, h_{\ep_n} \rangle + \langle \la - \rho, h_{\del_m} \rangle - (N-2m-2)\langle \ep_n, h_{\ep_n} \rangle] \\ 
&= \frac{1}{2}[(N+2m-1) + 2b -2(N+2m-2)] \\ 
&= b -(N+2m-3)/2,
\end{split}
\end{align} 
so $w = 0.$ Hence, we have that $e_{\del_m}u = \prod_{i=1}^{m-1} (e_{\ep_n-\del_{i}}e_{\ep_n+\del_i})w = 0.$

 In the case of \S\ref{2.3} one argues as in the case of \S\ref{2.1}, noting that the argument here for $\ep_{n-1} + \ep_n$ is similar to the one for $\ep_{i} - \ep_{i+1}$ in the case of \S \ref{2.1}.

The case of \S\ref{2.4} is similar to that of \S\ref{2.2}, except for the roots $2\del_m$ and $\ep_{n-1} - \ep_n.$ We have
\begin{align}
\begin{split}
e_{2\del_m}u =  &c_6 \prod_{i=1}^{m-1} (e_{\ep_{n}-\del_{i}}e_{\ep_n+\del_i})e_{\ep_n +\del_m}^{2}\prod_{i=1}^m (e_{\ep_{n-1}-\del_{i}}e_{\ep_{n-1}+\del_i}) f_{\ep_{n-1}+\ep_n}^{N+2m} v^+ \\
&+ c_7  \prod_{i=1}^m (e_{\ep_{n}-\del_{i}}e_{\ep_n+\del_i})\prod_{i=1}^{m-1} (e_{\ep_{n-1}-\del_{i}}e_{\ep_{n-1}+\del_i})e_{\ep_{n-1} + \del_m}^{2} f_{\ep_{n-1}+\ep_n}^{N+2m} v^+ = 0,
\end{split}
\end{align}
\begin{align}
\begin{split}
e_{\ep_{n-1} - \ep_n}u = &\sum_{k=1}^{m}  d_k \prod_{i=1}^{k-1} (e_{\ep_{n}-\del_{i}}e_{\ep_n+\del_i})e_{\ep_n - \del_k}\prod_{i=k+1}^{m} (e_{\ep_{n}-\del_{i}}e_{\ep_n+\del_i})\\
&\cdot\prod_{i=1}^{k-1} (e_{\ep_{n-1}-\del_{i}}e_{\ep_{n-1}+\del_i})e_{\ep_{n-1} - \del_k}e_{\ep_{n-1} + \del_k}^{2}\prod_{i=k+1}^{m} (e_{\ep_{n-1}-\del_{i}}e_{\ep_{n-1}+\del_i}) f_{\ep_{n-1}+\ep_n}^{N+2m} v^+ \\ 
&+ \sum_{k=1}^{m}  g_k \prod_{i=1}^{k-1} (e_{\ep_{n}-\del_{i}}e_{\ep_n+\del_i})e_{\ep_n + \del_k}\prod_{i=k+1}^{m} (e_{\ep_{n}-\del_{i}}e_{\ep_n+\del_i}) \\
&\cdot\prod_{i=1}^{k-1} (e_{\ep_{n-1}-\del_{i}}e_{\ep_{n-1}+\del_i})e_{\ep_{n-1} - \del_k}^{2}e_{\ep_{n-1} + \del_k}\prod_{i=k+1}^{m} (e_{\ep_{n-1}-\del_{i}}e_{\ep_{n-1}+\del_i}) f_{\ep_{n-1}+\ep_n}^{N+2m} v^+ = 0.
\end{split}
\end{align}

In the case of \S\ref{2.5}, we have that
\begin{align}
\begin{split}
e_{\ep_1 - \ep_2}u = &c_8  e_{+---}e_{+--+}e_{+-++}e_{++--}^{2}e_{++-+}e_{+++-}e_{++++} f_{\delta}^{N+4} v^+ \\ 
&+ c_9   e_{+---}e_{+--+}e_{+-+-}e_{++--}e_{++-+}^{2}e_{+++-}e_{++++} f_{\delta}^{N+4} v^+ = 0.
\end{split}
\end{align}
The case of $e_{\ep_2 - \ep_3}$ is entirely similar. We also have
\begin{align}
\begin{split}
e_{\ep_3}u = &c_{10}  e_{+--+}^{2}e_{+-+-}e_{+-++}e_{++--}e_{++-+}e_{+++-}e_{++++} f_{\delta}^{N+4} v^+ \\
&+ c_{11} e_{+---}e_{+--+}e_{+-++}^{2}e_{++--}e_{++-+}e_{+++-}e_{++++} f_{\delta}^{N+4} v^+ \\
&+ c_{12} e_{+---}e_{+--+}e_{+-+-}e_{+-++}e_{++-+}^{2}e_{+++-}e_{++++} f_{\delta}^{N+4} v^+ \\
&+ c_{13} e_{+---}e_{+--+}e_{+-+-}e_{+-++}e_{++--}e_{++-+}e_{++++}^{2} f_{\delta}^{N+4} v^+ = 0,
\end{split}
\end{align}
 and 
\[ e_{+---}u = e_{+---}^{2}e_{+--+}e_{+-+-}e_{+-++}e_{++--}e_{++-+}e_{+++-}e_{++++} f_{\delta}^{N+4} v^+ = 0. \]

In the case of \S\ref{2.6}, we have that
\begin{align}
\begin{split}
e_{\ep_1}u = &c_{14} e_{\del - \ep_2}e_{\del + \ep_2}e_{\del - \ep_3}e_{\del + \ep_3}e_{\del + \ep_1} e_{\del}f_{\del}f_{2\delta}^{M+3} v^+ \\
&+ c_{15} e_{\del - \ep_2}^{2}e_{\del + \ep_2}e_{\del - \ep_3}e_{\del - \ep_1}e_{\del + \ep_1} f_{\del}f_{2\delta}^{M+3} v^+ \\
&+ c_{16} e_{\del - \ep_2}e_{\del - \ep_3}^{2}e_{\del + \ep_3}e_{\del - \ep_1}e_{\del + \ep_1} f_{\del}f_{2\delta}^{M+3} v^+ \\
&+c_{17} e_{\del - \ep_2}e_{\del + \ep_2}e_{\del - \ep_3}e_{\del + \ep_3}e_{\del - \ep_1} e_{\del + \ep_1}e_{-\del + \ep_1}f_{2\delta}^{M+3} v^+ \\
&= 0 + 0 + 0 + c_{18} e_{\del - \ep_2}e_{\del + \ep_2}e_{\del - \ep_3}e_{\del + \ep_3}e_{\del - \ep_1}e_{-\del + \ep_1}^{2}f_{2\delta}^{M+2} v^+ \\
&= 0.
\end{split}
\end{align}
Similarly, we have that
\begin{align}
\begin{split}
e_{\del + \ep_3}u = &\pm e_{\del + \ep_3}^{2}e_{\del - \ep_2}e_{\del + \ep_2}e_{\del - \ep_3}e_{\del + \ep_1}e_{\del - \ep_1} f_{\delta}^{2M+7} v^+ = 0,\\
e_{\ep_2 - \ep_1}u = &c_{19} e_{\del - \ep_1}e_{\del - \ep_2}e_{\del + \ep_2}^{2}e_{\del - \ep_3}e_{\del + \ep_3} f_{\delta}^{2M+7} v^+ \\
&+ c_{20}  e_{\del - \ep_1}^{2}e_{\del + \ep_1}e_{\del + \ep_2}e_{\del - \ep_3}e_{\del + \ep_3} f_{\delta}^{2M+7} v^+ = 0.
\end{split}
\end{align}
Hence, the vectors $u$ are singular.
\end{proof}

\section{Homomorphisms between Verma Modules}\label{S4}

As an application of the results of \S \ref{S3}, we give an alternative proof of the following special case of the theorem of \cite{KK} presented in the introduction.

\begin{thm}\label{4}
Let $\mu \in \mathfrak{h}^*$ be such that $\langle \mu, h_{\al} \rangle = C \in \Z_{\geq 0}$, where $\al$ is a positive nonisotropic root that is not twice another root and $C$ is odd or zero if $\al$ is an odd root. There exists a nonzero homomorphism $M(s_{\al}\mu) \rightarrow M(\mu),$ in the following settings:
\begin{enumerate}
\item $\al = \del_p,$ $1 \leq p \leq m$ in the case of \S\ref{2.1},
\item $\al =\ep_q$, $1 \leq q \leq n$ in the case of \S\ref{2.2},
\item $\al = 2\del_p$, $1 \leq p \leq m$ in the case of \S\ref{2.3},
\item $\al = \ep_i + \ep_j$, $1 \leq i < j \leq n$ in the case of \S\ref{2.4},
\item $\al = \del$ in the case of \S\ref{2.5},
\item $\al = \del$ in the case of \S\ref{2.6}.
\end{enumerate}
\end{thm}

\begin{proof}
Let $C \in \Z_{> 0},$ the case of $C = 0$ being trivial. We first consider $\al = 2\del_p$ in the case of \S\ref{2.3}. Let $\beta$ be a positive even root.  Define a Shapovalov element, $\theta_{\beta,C} \in U(\mathfrak{b}^-)^{-C\beta}$, to be an element such that $\theta_{\beta,C}v^+ \in M(\nu)$ is singular for all $\nu$ such that $\langle \nu, h_{\beta} \rangle = C.$ Note that this is less restrictive than some other definitions (see \cite[\S4.12]{H} and \cite[Chapter 9]{M}).  The element of $U(\mathfrak{n}^-)$ which $\theta_{\beta,C}$ acts as on $M(\mu)$ is written as $\theta_{\beta,C}(\mu).$ Theorem \ref{3} shows the existence of a Shapovalov element $\theta_{2\del_m,C}$ for each positive integer $C$.

We will construct Shapovalov elements for the other $2\del_i$ by downward induction on $i$, following the proof of \cite[Theorem 9.2.6]{M}. Musson's argument itself follows a proof from \cite{S} in the case of semisimple Lie algebras.

Suppose that we have Shapovalov elements $\theta_{2\del_i,C}$ for a fixed $i$ and all $C \in \Z_{>0}.$  Write $\kappa = \del_{i-1} - \del_i,$ and $\mu = s_{\kappa}\nu.$  Suppose that $\langle \mu, h_{\kappa} \rangle = p \in \Z_{>0}$ and that $\langle \mu, h_{2\del_i} \rangle = C.$  We will show that there is a unique $\theta \in U(\mathfrak{n^-})^{-2C\del_{i-1}}$ such that 

\begin{equation}\label{4.1}
e_{-\kappa}^{p+2C}\theta_{2\del_i,C}(\mu) = \theta e_{-\kappa}^{p}.
\end{equation}

Write $L_{e_{-\kappa}}$ and $R_{e_{-\kappa}}$ for the endomorphisms of $U(\mathfrak{n}^-)$ given by left and right multiplication by $e_{-\kappa}.$  Because $\mathrm{ad}_{e_{-\kappa}}$ is locally nilpotent on $U(\mathfrak{n}^-)$, there exists a $k \in \Z_{\geq0}$ such that for all $l \in \Z_{\geq0}$

\begin{align*}
e_{-\kappa}^{l}\theta_{2\del_i,C}(\mu) &= L_{e_{-\kappa}}^{l}\theta_{2\del_i,C}(\mu) \\
&= (R_{e_{-\kappa}} + \mathrm{ad}_{e_{-\kappa}})^{l}\theta_{2\del_i,C}(\mu) \\
&= \sum_{i=0}^{l}\left( \begin{array}{c} l \\ i \end{array} \right) R_{e_{-\kappa}}^{l-i}(\mathrm{ad}_{e_{-\kappa}})^{i}\theta_{2\del_i,C}(\mu) \\
&= \sum_{i=0}^{k}\left( \begin{array}{c} l \\ i \end{array} \right)((\mathrm{ad}_{e_{-\kappa}})^{i}\theta_{2\del_i,C}(\mu))e_{-\kappa}^{l-i}.
\end{align*}

Hence, for sufficiently large $l$, $e_{-\kappa}^{l}\theta_{2\del_i,C}(\mu)v^+ \in U(\mathfrak{n}^-)e_{-\kappa}^{p}v^+ = M(\nu).$

By $\mathfrak{sl}_{2}$ relations,

\begin{align*}
e_{\kappa}e_{-\kappa}^{l}\theta_{2\del_i,C}(\mu)v^+ &= le_{-\kappa}^{l-1}(h_{\kappa} - l + 1)\theta_{2\del_i,C}(\mu)v^+ \\
&= l(p + 2C - l)e_{-\kappa}^{l-1}\theta_{2\del_i,C}(\mu)v^+.
\end{align*}

This implies that $e_{-\kappa}^{p + 2C}\theta_{2\del_i,C}(\mu)v^+ \in M(\nu).$  Since $M(\nu)$ is a free $U(\mathfrak{n}^-)$ module, there is a $\theta$ such that \eqref{4.1} holds, and because $e_{-\kappa}$ is not a zero divisor in $U(\mathfrak{n}^-),$ that $\theta$ is unique. Also, note that $\theta e_{-\kappa}^{p}v^+$ is singular.

Because, in \eqref{4.1}, $\theta_{2\del_i,C}$ is polynomial in $\mu$, and $\nu$ and $\mu$ are linearly related, $\theta$ is polynomial in $\nu$ such that $\langle \nu, h_{\kappa} \rangle$ is a negative integer, and $\langle \nu, h_{2\del_{i-1}} \rangle = C$.  As such $\nu$ form a Zariski dense subset of $\nu$ such that $\langle \nu, h_{2\del_{i-1}} \rangle = C$, the $\theta$ lift to a Shapovalov element $\theta_{2\del_{i-1},C} \in U(\mathfrak{b}^-),$ i.e. $\theta_{2\del_{i-1},C}(\nu) = \theta.$  This completes the induction and the proof of the theorem in the case of \S\ref{2.3}.

The proof of (1) in of the case of \S\ref{2.1} is similar, with $\del_i$ playing the role of $2\del_i$ and $\del_{i-1} -\del_i$ still playing the role of $\kappa.$

The proof of (2) in the case of \S\ref{2.2} is by downward induction on the index of $\ep_i$ ($\ep_i$ playing the role that $2\del_i$ played in the proof of (3)), the case $i=n$ following immediately from Theorem \ref{3}.  Moreover, $\ep_{i-1} - \ep_i$ plays the role of $\kappa$ here.

Finally, the proof of (4) in the case of \S\ref{2.4} is by downward induction on the sum of the indices of $\ep_i + \ep_j,$ ($\ep_i + \ep_j$ playing the role that $2\del_i$ played in the proof of (3)) the case of $i + j = 2n -1,$ following immediately from Theorem \ref{3}.  In this case $\ep_{i-1} - \ep_i$ or $\ep_{j-1} - \ep_j$ plays the role of $\kappa.$

Finally, (5) and (6) follow immediately from Theorem \ref{3}.
\end{proof}

\begin{rem}\label{4.2}
In the cases of \S\ref{2.1}, \S\ref{2.2} and \S\ref{2.3}, there remains a $W'$-orbit of positive nonisotropic roots not covered by Theorem \ref{4} or \cite[Corollary 9.2.7]{M} : $\al = \del_i + \del_j, 1 \leq i < j \leq m$ in the cases of \S\ref{2.1} and \S\ref{2.3} and $\al = \ep_k + \ep_l, 1 \leq k < l \leq n$ in the case of \S\ref{2.2}.  All other $W'$-orbits in the cases considered are covered by either of these results.
\end{rem}

\begin{rem}\label{4.3}
Retain the notation of Theorem \ref{4}.  In contrast to the semisimple Lie algebra setting, we do not know whether or not $\mathrm{dim}(\mathrm{Hom}(M(s_{\al}\mu), M(\mu))) = 1.$ Moreover, nonzero homomorphisms $M(s_{\al}\mu) \rightarrow M(\mu)$ are not necessarily injective (cf. \cite[Example 9.3.4]{M}).
\end{rem}

\begin{rem}
After Weiqiang Wang announced the results of this paper in an AMS meeting in 2017, Vera Serganova outlined to him an alternative approach of constructing singular vectors of a different form from ours using odd reflections. However, one would need to check, as we did for our candidate singular vectors in this paper, whether her candidate singular vectors are nonzero.  We were told that this odd reflection method is unable to handle the cases that we were unable to solve in this paper (see Remark \ref{4.2}). It should also be noted that Serganova's use of odd reflections resembles the approach of Musson's \cite[Corollary 9.3.6]{M} super generalization of Verma's theorem in the case of Verma modules of typical highest weight, as well as his approach to constructing singular vectors associated to isotropic odd roots \cite[Theorem 4.8]{M1}.
\end{rem}

{\bf Acknowledgement.} The author would like to thank his advisor, Weiqiang Wang, for formulating the conjecture on which this paper is based and for his discussions and advice. He would also like to thank the anonymous referee who pointed out the proof given in the second paragraph of the introduction using \cite{KK} and a density argument. This research was partially supported by NSF grant DMS-1702254.

\end{document}